\documentclass[a4paper,10pt]{amsart}
\usepackage{amsmath,amssymb,amsthm,amscd}
\usepackage{mathtools}
\usepackage{url}
\usepackage{color,hyperref}
\usepackage{stmaryrd}      
\theoremstyle{plain}
\newtheorem{theorem}{Theorem}[section]
\newtheorem{proposition}[theorem]{Proposition}
\newtheorem{lemma}[theorem]{Lemma}
\newtheorem{corollary}[theorem]{Corollary}

\theoremstyle{definition}
\newtheorem{definition}[theorem]{Definition}
\newtheorem{remark}[theorem]{Remark}

\newcommand{\Q}{\mathbb{Q}} 
\newcommand{\N}{\mathbb{N}} 
\newcommand{\Z}{\mathbb{Z}} 
\newcommand{\R}{\mathbb{R}} 
\newcommand{\sub}{\subseteq} 
\newcommand{\set}[2]{\{#1|\ #2\}} 
\newcommand{\conv}{\mathrm{conv}} 
\newcommand{\gen}[1]{\langle #1\rangle} 
\newcommand{\pure}[1]{\sqrt{#1}} 
\newcommand{\ri}{\mathrm{ri}} 
\newcommand{\x}{\textit{\textbf{x}}} 

\begin{document}

\title[Idempotence of finit. gen. commutative semifields]{Idempotence of finitely generated \\ commutative semifields}

\author[V.~Kala]{V\'{i}t\v{e}zslav~Kala}

\address{Charles University, Faculty of Mathematics and Physics, Department of Algebra, Sokolovsk\'{a} 83, 186 75 Prague 8, Czech Republic}

\address{University of G\"{o}ttingen, Mathematisches Institut, Bunsenstr.~3-5, D-37073 G\"{o}ttingen, Germany}

\email{vita.kala@gmail.com}

\author[M.~Korbel\'{a}\v{r}]{Miroslav~Korbel\'{a}\v{r}}

\address{Department of Mathematics, Faculty of Electrical Engineering, Czech Technical University in Prague, Technick\'{a} 2, 166 27 Prague 6, Czech Republic}

\email{miroslav.korbelar@gmail.com}

\thanks{The first author was supported by Neuron Impulse award and by Charles University Research Centre
	program UNCE/SCI/022.}




\begin{abstract}
We prove that a commutative parasemifield $S$ is additively idempotent provided that it is finitely generated as a semiring. Consequently, every proper commutative semifield $T$ that is finitely generated as a semiring is either additively constant or additively idempotent. As part of the proof, we use the classification of finitely generated lattice-ordered groups to prove that a certain monoid associated to the parasemifield $S$ has a distinguished geometrical property called prismality.
\end{abstract}

\maketitle

\section{Introduction}

It is easy to see that the field $\mathbb Q$ of rational numbers is not finitely generated as a ring. More generally, 
a \emph{folklore} theorem (perhaps due to Kaplansky) states that if a commutative ring is simple and finitely generated, then it is finite. Of course, such rings are precisely the finite fields $\mathbb F_q$ and the zero-multiplication rings $\Z_p$ of prime order. This result can also be viewed as a classification of all finitely generated simple commutative rings.

Semirings often behave similarly as rings (see e.g., \cite{golan} for an overview). Thus it is natural to ask whether a similar result as  above  also holds in the more general setting of semirings. 
In this paper we will not consider the other natural generalization to non-commutative rings; in fact, all algebraic structures will be commutative throughout the paper. Nevertheless, let us mention at least one of the latest results on simple non-commutative rings \cite{lam}.

\begin{definition}
 Recall that by a (commutative) \emph{semiring} we mean a non-empty set $S$ equipped with two associative and commutative operations (addition and multiplication) where the multiplication distributes over the addition from both sides. 
A semiring $S$ is a \emph{semifield} if it contains a zero element $0$ and the set $S\setminus\{0\}$ is a group with respect to the multiplication. In the case that the entire multiplicative part $S(\cdot)$ is a group, the semiring $S$ is called a \emph{parasemifield}. A semiring (a semifield, resp.) is called \emph{proper} if it is not a ring. Finally, a semiring $S$ is \emph{additively idempotent} if $x+x=x$ for all $x\in S$ and  \emph{additively constant} if the map $S\times S\to S$, $(x,y)\mapsto x+y$ is constant. 
\end{definition}

In mathematics, semirings and semifields are ubiquitous,  which makes them one of the fundamental algebraic objects. Perhaps the first mathematical structure one encounters, the set of natural numbers $\mathbb N$, is a semiring. Besides this obvious observation, semirings and semifields play an important role in modern  mathematics as well as in a wide range of applications.
Let us mention a few of them:
Tropical geometry, which is essentially algebraic geometry over additively idempotent semirings, is useful in studying piecewise linear functions in optimization problems (e.g., \cite{Ga, ims, ir} and the references therein). Tropical semirings are also used in constructing cluster algebras \cite{keller} and appear in the process of so called dequantization \cite{litvinov}.
In number theory, Connes and Consani \cite{cc,cc2} were motivated by the goal of working over the ``field of one element" \cite{F1_geometry} (related to semirings) and extended this viewpoint further with a certain hope of proving Riemann hypothesis.
Their work was recently generalized by Leichtnam \cite{Le} to cover more general additively idempotent semifields.
An interesting direction is also the study of cryptography based on semirings, as developed by Maze, Monico, Rosenthal, Zumbr\"{a}gel, and others \cite{mmr,monico,zumbragel}. It could help in coping with some of the vulnerabilities of classical cryptography based on modular arithmetic. {Semirings are also important for weighted automata in theoretical computer science \cite{droste}.}
Yet another class of applications arises thanks to the correspondence between certain semifields, lattice-ordered groups, and MV-algebras. These provide useful tools in multi-valued logic \cite{BDN-C, BDNF-B, BDNF, DNG, dinola, M, mundici}.  For further applications and references, see e.g. \cite{golan, IKN, KNZ}.

\

In this paper we are interested in studying finitely generated simple semirings.
Unfortunately, the situation quickly becomes more convoluted than in the case of rings. 
First of all, ideals in semirings do not correspond to congruences, and so one has to distinguish between \emph{congruence-} and \emph{ideal-simple} semirings (i.e., those that have only the trivial congruences, and those in which there are no proper ideals).
Both of these cases were (almost completely) classified by Bashir, Hurt, Jan\v{c}a\v{r}\'{i}k, and Kepka \cite{simple}. It has turned out that there are semirings which are  finitely generated, both congruence- and ideal-simple, and yet \emph{infinite}  -- for instance, the additively idempotent tropical semiring $\mathbb Z(\oplus, \odot)$ with the semiring addition $a\oplus b=\min(a, b)$ and multiplication $a\odot b=a+b$. On the other hand, every \emph{finite} congruence- or ideal-simple \emph{proper} semiring is either additively constant or additively idempotent \cite{simple}.

Hence we need to modify the \emph{folklore} theorem to also deal with these cases. For congruence-simple semirings, this quite easily follows from the classification using \cite[Corollary 14.3]{simple}: 
Every proper finitely generated congruence-simple
semiring is either additively constant or additively idempotent.

The main result of this paper is the proof of an analogous result for ideal-simple semirings:

\begin{theorem}\label{main thm}
(a) Every parasemifield that is finitely generated as a semiring is additively idempotent.

(b) Every proper semifield that is finitely generated as a semiring is either additively constant or additively idempotent.

(c) Every proper finitely generated ideal-simple semiring is either additively constant or additively idempotent.
\end{theorem}

This statement can be now understood as an extension of the \emph{folklore} theorem referred in the beginning, i.e., that every (commutative) field that is finitely generated as a ring is finite. Of course, the greatest difference is the existence of additively idempotent semifields.

Since every semifield is clearly ideal-simple, part (b) of the theorem follows immediately from (c). 
Also,
one can be slightly more precise in the additively constant case: this occurs if and only if there is a finitely generated (multiplicative) abelian group $G(\cdot)$ and the semiring is the semifield 
$S:=G\cup \{o\}$, where $o$ is  a new element. Operations that extend the multiplication on $G$ are defined by $a+ b=o$ and $a\cdot o=o$ for all $a, b\in S$.

\

Theorem \ref{main thm} was at first formulated as a conjecture in \cite{simple} for the infinite cases. Our version is a slight modification that considers \emph{proper} semirings instead of \emph{infinite} ones and that includes in this way also the finite cases, whose properties were mentioned above.
The equivalence of (a) and (c) was then established by Je\v zek, Kala, and Kepka \cite{jk,a note}, and so it remained to prove (a). 
Initial steps in this direction were done by the present authors and Kepka \cite{notes}, and continued together with Je\v zek \cite{jezek}, where part (a) was proved in the case of 2 generators. Note that there are no parasemifields that are 1-generated as semirings \cite[Remark 4.22]{notes}.

As we already mentioned, the goal of this paper is to prove Theorem \ref{main thm} in general (by proving Theorem \ref{main-theorem}) 
and to settle in this way a problem that remained open for more than fifteen years.  
The general idea of the proof uses a suitable subsemiring $Q_S$ of the parasemifield $S$ and an associated monoid $\mathcal C_X(S)\subseteq \mathbb N_0^n$ (see Section \ref{section 2} for the definitions). In the 2-generated case \cite{jezek}, the monoid $\mathcal C_X(S)$ was simple enough to consider only  elementary properties of the geometry to prove the theorem. 
However, in the general case the monoid has a much more complicated shape, and so the situation is significantly harder.

One key ingredient in our proof comes from the classification of finitely generated lattice-ordered groups ($\ell$-groups) by Busaniche, Cabrer, and Mundici \cite{BCM}. 
There is a well-known term-equivalence between $\ell$-groups and additively idempotent parasemifields, and so
Kala \cite{l-groups} has recently used their results to obtain a classification of additively idempotent parasemifields which are finitely generated as semirings (see Definition \ref{rooted_tree} below). 
The monoid $\mathcal C_X(T)$ associated to each of these additively idempotent parasemifields $T$ has a special geometric property, \textit{prismality} (introduced in Section \ref{section 1}). Now given a general parasemifield $S$ that is finitely generated as a semiring, we consider its largest factor-pa\-ra\-semi\-field that is additively idempotent.
The associated monoids of both these parasemifields are the same, and so we conclude that the monoid of $S$ is also prismal. This then gives us the crucial missing geometrical information that allows us to finish the proof of Theorem \ref{main-theorem} (i.e., to confirm  Conjecture  \ref{main thm}).

As in the case of rings, these results then imply a classification of all finitely generated ideal-simple semirings:
 if such a semiring is, moreover, a parasemifield then it must be one from Definition \ref{rooted_tree}.
 The extension to general ideal-simple semirings is then a routine application of the classification theorems of \cite{simple,jk}, and so we don't state it explicitly here.
Note that an analogous result was recently proved by Schneider and Zumbr\"{a}gel for simple compact (not necessarily commutative) semirings \cite{sz}.

\

As for the organization of the paper, in the second section we introduce a property of submonoids of $(\Z^n,+)$ called \emph{prismality} and study its basic properties. In the third section we prove in Theorem \ref{prismality} that every monoid $\mathcal C_X(S)$ associated to a finite set $X$, that generates a parasemifield $S$ as a semiring, is prismal. Finally, the fourth section uses this result to prove  Theorem \ref{main-theorem}. 

Let us conclude this introduction by pointing out 
that the ideas presented in this paper can probably be generalized to the situation of additively divisible semirings.
Another very interesting generalization of our results and methods is to apply them to the Banach semifield setting of Leichtnam \cite{Le}. This should (hopefully) allow us to generalize and extend his results (e.g., to remove the Assumption 2) -- we also plan to study this in the (near) future.

\section{Prismal monoids}\label{section 1}

 In this section we define a property of  submonoids  of $(\Z^n,+)$ called \emph{prismality} and study its behaviour. 

First, we need some definitions. Every vector space considered in this paper is assumed to be a \emph{real} vector space.  A vector subspace $V\sub\R^{n}$ is said \emph{to be defined over $\Q$} if $V$ has a basis that consists of vectors from $\Q^{n}$.  Of course, this is equivalent to assuming that $V$ has a basis of vectors from $\Z^n$.

Let $M$ be a subset of $\R^{n}$. By $\gen{M}$ we denote the \emph{vector subspace of $\R^{n}$ generated by $M$}, by  $\overline{M}^{\R^{n}}$ the usual \emph{topological closure
of $M$} and by $\conv(M)$ the \emph{convex hull of $M$}. Further, we denote by $\dim(M)$ the \emph{dimension of the convex hull} $\conv(M)$. Note that in the case when $M$ is a cone or a monoid, this is the dimension of the vector space $\gen{M}$. 

By a \emph{cone} $K\sub\R^{n}$ we mean a convex set such that for every non-negative real number $\lambda$ and every $u\in K$ we have $\lambda u\in K$. When working with a convex set $A\sub\R^{n}$ we will use the notion of the \emph{relative interior of} $A$, denoted as $\ri(A)$, that is defined as the interior of $A$ with respect to the affine hull of $A$.

\begin{definition}\label{def}
For a submonoid $\mathcal{C}$ of $(\Z^n,+)$ define its \emph{closure} $\overline{\mathcal{C}}$ as $$\overline{\mathcal{C}}=\Z^n\cap \overline{\conv(\mathcal{C})}^{\R^{n}}\ .$$ 

Further put $$\pure{\mathcal{C}}=\set{\alpha\in \Z^n}{(\exists k\in\N)\ k\alpha\in\mathcal{C}}.$$

We say that a monoid $\mathcal{C}\sub \Z^n$ is
\begin{itemize}
 \item {\it pure} if $\pure{\mathcal{C}}=\mathcal{C}$.
 \item {\it almost prismal} if for every vector subspace $V$ of $\R^n$, the monoid $\overline{\mathcal{C}\cap V}$ is finitely generated.
 \item {\it prismal} if it is pure and almost prismal.
\end{itemize}
\end{definition}

Finitely generated monoids are often called \emph{affine} (although we will not use this terminology). 
As is clear from the definition, almost prismal monoids are a generalization of affine monoids.

 We will be interested in properties of prismal monoids, namely whether they are closed under intersections, products, and homomorphic images. The answer is ``Yes!", 
but the arguments are not entirely easy and one has to be a little careful, as for example Remark \ref{remark 1.5} shows.

Before we can prove the closedness in Theorems \ref{change}, \ref{cartesian}, and \ref{epi_prismal}, we will need some technical results on cones and monoids. These are essentially known, but not easily located in the literature, so we also include (most of) their proofs.

\begin{proposition}\label{convex_0}\cite[Theorems 6.3 and 6.5]{relative_interior}
 Let $K, L\sub\R^{n}$ be cones. Then
 \begin{enumerate}
  \item\label{1} $\ri\left(\overline{K}^{\R^n}\right)= \ri(K)$  and $\overline{K}^{\R^{n}}= \overline{\ri(K)}^{\R^{n}}$.
  \item\label{1.5} $\ri(K\cap L)=\ri(K)\cap\ri(L)$ if $\gen{K}=\gen{L} =\gen{K\cap L}$.
 \end{enumerate}
\end{proposition}

\begin{proposition}\label{convex}
 Let $\mathcal{C},\mathcal{D}\sub\Z^{n}$ be pure monoids. Then
 \begin{enumerate}
  \item\label{5} $\Z^{n}\cap\conv(\mathcal{C})=\mathcal{C}$.
  \item\label{4}  $\conv(\mathcal{C})\cap\conv(\mathcal{D})=\conv\big(\mathcal{C}\cap \conv(\mathcal{D})\big)=\conv(\mathcal{C}\cap\mathcal{D})$.
 \end{enumerate}
\end{proposition}

\begin{proof}

(i) Every $\alpha\in\Z^{n}\cap\,\conv(\mathcal{C})$ is a convex linear combination of a finite affinely independent subset of $\mathcal{C}$. Hence the coefficients in this combination have to be rational and due to the convexity also non-negative. Thus there is $k\in\N$
 such that $k\alpha\in\mathcal{C}$.  Since $\mathcal{C}$ is pure, we obtain that $\alpha\in\mathcal{C}$. The other inclusion is obvious.

 (ii) By (i), we see that $\mathcal C\cap \conv(\mathcal D)=\mathcal C\cap (\conv(\mathcal D)\cap\Z^n)=\mathcal C\cap\mathcal D$. Hence it is enough to show that $\conv(\mathcal{C})\cap\conv(\mathcal{D})=\conv(\mathcal{C}\cap\mathcal{D})$. This assertion holds if $\mathcal{C}$ and $\mathcal{D}$ are finitely generated (see \cite{bruns}). In the general case let $y\in\conv(\mathcal{C})\cap\conv(\mathcal{D})$. Then $y\in\conv(\mathcal{C}')\cap\conv(\mathcal{D}')$ for some finitely generated submonoids $\mathcal{C}'\sub\mathcal{C}$ and $\mathcal{D}'\sub\mathcal{D}$. Hence we have $y\in\conv(\mathcal{C}'\cap\mathcal{D}')\sub \conv(\mathcal{C}\cap\mathcal{D})$. The other inclusion is obvious.
\end{proof}

\begin{proposition}\label{closure_basic}
  Let $\mathcal{C}, \mathcal{D}\sub\Z^{n}$ be pure monoids and $V$ a vector subspace of $\R^{n}$. Then $$\overline{(\mathcal{C}\cap\mathcal{D})\cap V}=\overline{\mathcal{C}\cap W}\cap\overline{\mathcal{D}\cap W}$$ where $W=\gen{\mathcal{C}\cap\mathcal{D}\cap V}$.
\end{proposition}
\begin{proof}
Set  $\mathcal{C}'=\mathcal{C}\cap W$ and $\mathcal{D}'=\mathcal{D}\cap W$.  Then $(\mathcal{C}\cap\mathcal{D})\cap V=\mathcal{C}'\cap \mathcal{D}'$. Hence $\gen{\mathcal{C}'\cap\mathcal{D}'}=W$ and therefore we also have $\gen{\mathcal{C}'}=W=\gen{\mathcal{D}'}$. Put $M=\overline{\conv(\mathcal{C}')}^{\R^{n}}$ and $N=\overline{\conv(\mathcal{D}')}^{\R^{n}}$. Then also $\gen M=\gen N=\gen{M\cap N}=W$.

By Proposition \ref{convex_0}, parts \ref{1}, \ref{1.5}, and Proposition \ref{convex}\ref{4}, we now obtain,
$$\ri(M\cap N)=\ri(M)\cap \ri(N)=\ri\left(\overline{\conv(\mathcal{C}')}^{\R^{n}}\right)\cap \ri\left(\overline{\conv(\mathcal{D}')}^{\R^{n}}\right)=$$
$$=\ri\big(\conv(\mathcal{C}')\big)\cap \ri\big(\conv(\mathcal{D}')\big)=\ri\big(\conv(\mathcal{C}')\cap \conv(\mathcal{D}')\big)=\ri\big(\conv(\mathcal{C}'\cap \mathcal{D}')\big).$$
Therefore
$$\overline{\conv(\mathcal{C}')}^{\R^{n}}\cap  \overline{\conv(\mathcal{D}')}^{\R^{n}}=M\cap N=\overline{M\cap N}^{\R^{n}}=\overline{\ri(M\cap N)}^{\R^{n}}=$$
$$=	\overline{\ri\big(\conv(\mathcal{C}'\cap \mathcal{D}')\big)}^{\R^{n}}=\overline{\conv(\mathcal{C}'\cap \mathcal{D}')}^{\R^{n}}\ .$$
Finally, we obtain the equality $$\overline{\mathcal{C}\cap W}\ \cap\ \overline{\mathcal{D}\cap W}=\Z^{n}\cap\overline{\conv(\mathcal{C}\cap W)}^{\R^{n}}\cap  \overline{\conv(\mathcal{D}\cap W)}^{\R^{n}}=$$
$$=\Z^{n}\cap\overline{\conv(\mathcal{C}\cap \mathcal{D}\cap V)}^{\R^{n}}=\overline{(\mathcal{C}\cap\mathcal{D})\cap V}\ .$$
\end{proof}

\begin{theorem}\label{change}
Almost prismal monoids are closed under intersections.
\end{theorem}
\begin{proof}
Let  $V$ be a vector subspace of $\R^n$ and $\mathcal{C}, \mathcal{D}$ be almost prismal monoids. By Proposition \ref{closure_basic}, $\overline{(\mathcal{C}\cap\mathcal{D})\cap W}=\overline{\mathcal{C}\cap W}\cap\overline{\mathcal{D}\cap W}$  is a finitely generated monoid  as $\overline{\mathcal{C}\cap W}$ and $\overline{\mathcal{D}\cap W}$ are finitely generated. The rest is clear.
\end{proof}

\begin{remark}\label{remark 1.5}
 Note that in general it need not be true that $\overline{\mathcal{C}\cap\mathcal{D}}=\overline{\mathcal{C}}\cap\overline{\mathcal{D}}$ (similarly as in the case of a usual topological closure operator). An example is when $m=2$ and $\mathcal{C}=\set{(i, j)\in\N^{2}_{0}}{i<j\ \text{or}\ i=j=0}$ and $\mathcal{D}=\set{(i, j)\in\N^{2}_{0}}{i>j\ \text{or}\ i=j=0}$. Then $\overline{\mathcal{C}\cap\mathcal{D}}=\{(0, 0)\}$, while $\overline{\mathcal{C}}\cap\overline{\mathcal{D}}=\set{(k, k)}{k\in\N_{0}}$.
Note that in this case $W=\gen{\mathcal{C}\cap\mathcal{D}}=\{(0, 0)\}$ is 0-dimensional.
\end{remark}

\begin{proposition}\label{simplification}
Let $\mathcal{C}\sub\Z^{n}$ be a monoid. Then
\begin{enumerate}
 \item\label{first} $\mathcal{C}$ is finitely generated if and only if $\pure{\mathcal{C}}$ is finitely generated.
 \item\label{second} $\mathcal{C}$ is almost prismal if and only if for every vector subspace $V\sub\R^{n}$ defined over $\Q$ the monoid $\overline{\mathcal{C}\cap V}$ is finitely generated.
\end{enumerate}
\begin{proof}
 It follows easily from the fact that a monoid $\mathcal{C}'$ is finitely generated if and only if the cone $\conv(\mathcal{C}')$ is finitely generated (as a cone) (see \cite{bruns}).
\end{proof}

\begin{lemma}\label{integral_basis}
 Let $V\sub\R^{n}$ be a vector subspace defined over $\Q$ and $\nu:V\to \R^{n}$ be a linear embedding such that $\nu(V\cap\Z^{n})\sub\Z^{n}$. Then $\pure{\nu(V\cap\Z^{n})}=\nu(V)\cap\Z^{n}$.
\end{lemma}
\end{proposition}
\begin{proof}
Since $V$ is defined over $\Q$, there is a basis $\{u_{1},\dots,u_{k}\}\sub\Z^{n}$ of $V$. Then $\{\nu(u_{1}),\dots,\nu(u_{k})\}\sub\Z^{n}$ is a basis of $\nu(V)$.  For  $\alpha\in\nu(V)\cap\Z^{n}$ there are integers $r_{i}\in\Z$, $i=1,\dots,k$, and $s\in\N$ such that $\alpha=\sum^{k}_{i=1}\frac{r_{i}}{s}\ \nu(u_{i})$. Hence $s\alpha=\nu\big(\sum^{k}_{i=1}r_{i}\ u_{i}\big)\in\nu(V\cap\Z^{n})$ and $\alpha\in\pure{\nu(V\cap\Z^{n})}$. The rest is obvious.
\end{proof}

\begin{proposition}\label{alm_prism_equiv}
Let $\mathcal{C}\sub\Z^{n}$ be a monoid and $V\sub\R^{n}$ be a vector subspace. If $\mathcal{C}$ is prismal then the monoid $\mathcal{C}\cap V$ is prismal too.

 Let $\mathcal{C}\sub V$ and let $V$ be defined over $\Q$. If $\nu:V\to \R^{n}$ is a linear embedding such that $\nu(V\cap\Z^n)\sub\Z^n$, then the monoid $\mathcal{C}$ is almost prismal if and only if the monoid $\nu(\mathcal{C})$ is almost prismal.
\end{proposition}
\begin{proof}
 The first claim is obvious. Let now $W$ be a vector subspace of $\R^{n}$. Set $U=\nu^{-1}(W)\sub V$. Then $\nu(\mathcal{C}\cap U)=\nu(\mathcal{C})\cap\nu(U)=\nu(\mathcal{C})\cap W$. By Lemma \ref{integral_basis}, we know that $\nu(V)\cap\Z^{n}=\pure{\nu(V\cap\Z^{n})}$. Hence we have
 $$\overline{\nu(\mathcal{C})\cap W}=\overline{\nu(\mathcal{C}\cap U)}=\Z^{n}\cap\overline{\conv\big(\nu(\mathcal{C}\cap U)\big)}^{\R^{n}}=$$
 $$=\nu(V)\cap\Z^{n}\cap\overline{\nu\big(\conv(\mathcal{C}\cap U)\big)}^{\R^{n}}=\pure{\nu(V\cap\Z^{n})}\cap\nu\left(\overline{\conv(\mathcal{C}\cap U)}^{\R^{n}}\right)=$$
 $$=\pure{\nu(V\cap\Z^{n})}\cap\pure{\nu\left(\overline{\conv(\mathcal{C}\cap U)}^{\R^{n}}\right)}=$$
 $$=\pure{\nu\left(V\cap\Z^{n}\cap\overline{\conv(\mathcal{C}\cap U)}^{\R^{n}}\right)}=\pure{\nu\left(\overline{\mathcal{C}\cap U}\right)}=\pure{\nu\left(\overline{\mathcal{C}\cap \nu^{-1}(W)}\right)}.$$

 Now, the monoid $\overline{\nu(\mathcal{C})\cap W}$ is finitely generated if and only if $\pure{\nu\left(\overline{\mathcal{C}\cap \nu^{-1}(W)}\right)}$ is so. And this happens if and only if the monoid $\overline{\mathcal{C}\cap \nu^{-1}(W)}$ is finitely generated, by Lemma \ref{simplification}\ref{first}.

Therefore $\nu(\mathcal{C})$ is almost prismal if and only if $\mathcal{C}$ is almost prismal.	
\end{proof}

\begin{theorem}\label{cartesian}
 A cartesian product of prismal monoids is a prismal monoid.
\end{theorem}
\begin{proof}
 Let $\mathcal{C}_i$ be prismal monoids in $\Z^{d_i}$ for $i=1,2$. Then the monoid $\mathcal{C}_1 \times\mathcal{C}_2\sub\Z^{d_1}\times\Z^{d_2}$ can be expressed as $\mathcal{C}_1 \times\mathcal{C}_2=(\mathcal{C}_1\times \Z^{d_2})\cap(\Z^{d_{1}} \times\mathcal{C}_2) $. In the view of Theorem  \ref{change}, it is therefore enough to prove that $\mathcal{C}\times\Z$ is prismal  whenever $\mathcal{C}$ is prismal.  
 
 Let $\mathcal{C}\sub\Z^{n}$ be a prismal monoid and let $\pi:\R^{n}\times\R\to\R^{n}$ be the projection forgetting the last component. Let $V$ be a subspace of $\R^{n}\times\R$. Due to Proposition \ref{simplification}\ref{second}, we may consider that $V$ is defined over $\Q$. 
 
 If $\ker(\pi)=\gen{(0,0,\dots,0,1)}\sub V$, then $V=W\oplus\R$ for some subspace $W$ of $\R^n\times\{0\}$. Therefore we clearly obtain  $\overline{V\cap(\mathcal{C}\times\Z)}=\overline{(W\cap\mathcal{C})\times\Z}=\overline{(W\cap\mathcal{C})}\times\Z$. Since $\mathcal{C}$ is prismal, $\overline{(W\cap\mathcal{C})}$ is a finitely generated monoid and the monoid $\overline{V\cap(\mathcal{C}\times\Z)}$ is finitely generated as well.

Let now  $\ker(\pi)\cap V=0$. Set $\mathcal{C}'=V\cap(\mathcal{C}\times\Z)$. By Proposition \ref{alm_prism_equiv}, the monoid $\mathcal{C}'$ is almost prismal if and only if the monoid $\pi(\mathcal{C}')=\pi(V)\cap\pi(\mathcal{C}\times\Z)=\pi(V)\cap\mathcal{C}$ is almost prismal. From the assumption we know that $\mathcal{C}$ is almost prismal, hence  $\pi(V)\cap\mathcal{C}$ is so and, consequently, $\mathcal{C}'$ is prismal as well. In particular, $\overline{\mathcal{C}'}=\overline{V\cap(\mathcal{C}\times\Z)}$ is finitely generated.

Finally, we have verified the conditions of prismality and the monoid $\mathcal{C}\times\Z$ is therefore prismal.
\end{proof}

\begin{theorem}\label{epi_prismal}
 Let $\pi:\R^{n}\to\R^{k}$ be a linear epimorphism such that $\pi(\Z^{n})\sub\Z^{k}$. If a monoid $\mathcal{C}\sub\Z^{k}$ is almost prismal then the monoid $(\pi_{|\Z^{n}})^{-1}(\mathcal{C})$ is almost prismal as well.
\end{theorem}
\begin{proof}
 First, assume that $\pi$ is a canonical projection of the form $\pi(e_i)=e_i$ for $i=1,\dots,k$ and $\pi(e_j)=0$ for $j=k+1,\dots,n$ (where $e_i$ are the standard basis vectors). Then $(\pi_{|\Z^{n}})^{-1}(\mathcal{C})=\mathcal{C}\times\Z^{n-k}$. By Theorem \ref{cartesian}, this monoid is prismal, provided that $\mathcal{C}$ is prismal.
 
 Further, let $\pi$ be a scaling such that  $n=k$ and $\pi(e_i)=k_i\cdot e_i$ for some $0\neq k_i\in\Z$, $i=1,\dots,n$. Clearly, $\pure{\pi(\Z^{n})}=\Z^{n}$. Set $\widetilde{\mathcal{C}}=(\pi_{|\Z^{n}})^{-1}(\mathcal{C})$. Using Proposition \ref{simplification}\ref{first} and Lemma \ref{integral_basis} we obtain that the monoid  $\widetilde{\mathcal{C}}$ is almost prismal if and only if $\pi(\widetilde{\mathcal{C}})$  is almost prismal. 
  This is equivalent to $\pure{\pi(\widetilde{\mathcal{C}})}=\pure{\pi\big((\pi_{|\Z^{n}})^{-1}(\mathcal{C})\big)}=\pure{\mathcal{C}}$ being almost prismal, and that happens if and only if $\mathcal{C}$ is almost prismal. Since $\mathcal{C}$ is almost prismal, we have proved that the monoid $\widetilde{\mathcal{C}}=(\pi_{|\Z^{n}})^{-1}(\mathcal{C})$ is almost prismal too.
 
 Finally, in the general case we can consider that $\pi=\psi_1\circ\nu\circ\widetilde{\pi}\circ\psi_2$, where $\widetilde{\pi}$ is a canonical projection, $\nu$ a scaling and $\psi_1(\psi_2$, resp.) corresponds to an isomorphism of the group $\Z^{k}$ ($\Z^{n}$, resp.). Now, combining all the cases together we obtain that from prismality of $\mathcal{C}$  it follows that $(\pi_{|\Z^{n}})^{-1}(\mathcal{C})$ is almost prismal, too.
\end{proof}

Now we will be interested in establishing a decomposition of a prismal monoid into faces:
Let $K\sub \R^{n}$ be a convex set. A non-empty subset $A\sub K$ will be called a \emph{relatively open face} of $K$ if 
\begin{itemize}
 \item $A$ is convex,
 \item $\ri(A)=A$ and
 \item for every line segment $L\sub K$ such that $\ri(L)\cap A\neq \emptyset$, we have $\ri(L)\sub A$.
\end{itemize}

\begin{theorem}\label{cone_decomp}
For every cone $K$  in $\R^n$ there is a (unique) decomposition $\mathbf{K}=\set{A_i}{i\in I}$ of $K$ into disjoint union of relatively open faces $A_i$ of $K$, i.e., $K=\bigsqcup_{i\in I} A_i$.

Moreover, let $A\in\mathbf{K}$ and let $x\in A$ and $y\in K\setminus\overline{A}^{\R^n}$. Then  there is a relatively open face $B\in \mathbf{K}$ such that the relatively interior of the line segment $\conv(\{x,y\})$ lies in $B$ and $\dim(B)>\dim(A)$.
\end{theorem}
\begin{proof}
It is easy to verify that the following construction provides the desired decomposition.
For every $x\in K$ there is a unique vector space $W_{x}$ of maximal dimension such that $x$ is a relatively inner point of the convex set $W_{x}\cap K$. Now, set a relation on $K$ as follows $x\sim y$  if and only if $W_x=W_y$. This relation is an equivalence and the partition sets are the desired relatively open faces of $K$. In particular, such a face $A$ is of the form $A=\ri(W_x \cap K)$, where $x\in A$.
\end{proof}

Finally, we are ready to prove the following result, which is the culmination of this section. The corollary establishes geometrical properties of prismal monoids that will play a key role later in the proof of Theorem \ref{nilpotent}. 
Its proof will go by downwards induction on the dimension of the face $\mathcal D$, and so we will need to be able to relate the properties of lower-dimensional faces to the higher dimensional ones, as in part (iii) of the  following corollary.

\begin{corollary}\label{monoid_decomp}
  Let $\mathcal{C}\sub\R^n$ be a prismal monoid. Let $K=\conv(\mathcal{C})$ and let $\mathbf{K}$ be the unique decomposition of $K$ into relatively open faces. For a relatively open face  $A\in\mathbf{K}$ of $K$ let $A^0=A\cup\{0\}$ be the cone arising from $A$.
  
Set $\mathbf{D}(\mathcal{C}):=\set{A^{0}\cap\mathcal{C}}{A\in\mathbf{K}}$. Then $\mathbf{D}(\mathcal{C})$ is a decomposition of $\mathcal{C}$ into pure monoids  (i.e., $\mathbf{D}(\mathcal{C})=\bigcup_{\mathcal D\in\mathbf{D}(\mathcal{C})} \mathcal D$ and the union is ``almost disjoint": $\mathcal D\cap\mathcal D'=\{0\}$ for $\mathcal D\neq \mathcal D'$) and for each $\mathcal{D}\in\mathbf{D}(\mathcal{C})$ we have:
\begin{enumerate}
 \item\label{i} The monoid $\overline{\mathcal{D}}$ is finitely generated.
 \item\label{ii}  If $\dim(\mathcal{D})=\dim(\mathcal{C})$ then $\overline{\mathcal{D}}=\overline{\mathcal{C}}$.
 \item\label{iii}   For all $0\neq\alpha\in \mathcal{D}$ and $\beta\in \mathcal{C}\setminus \overline{\mathcal{D}}$ there is  $\mathcal{E}\in\mathbf{D}(\mathcal{C})$ such that $\dim(\mathcal{E})>\dim(\mathcal{D})$  and $\alpha+\beta\in \mathcal{E}$. 
 \item\label{iv}  For all $0\neq\alpha\in \mathcal{D}$ and $\gamma\in \overline{\mathcal{D}}$ we have $\alpha+\gamma\in \mathcal{D}$.
\end{enumerate} 
\end{corollary}
\begin{proof}
 By the definition of the decomposition, we have $\mathcal{D}=A^{0}\cap\mathcal{C}$, where $A=\ri(W\cap K)$ for some vector subspace $W\sub\R^{n}$ defined over $\Q$. Clearly, $\mathcal{D}$ is a pure monoid.
 
 Further, we show that $\overline{\mathcal{D}}=\overline{W\cap\mathcal{C}}=\Z^{n}\cap\overline{A}^{\R^{n}}$. By the definition, we have 
 $$\overline{\mathcal{D}}=\Z^{n}\cap\overline{\conv(A^0\cap\mathcal{C})}^{\R^n}$$ and $$\overline{W\cap\mathcal{C}}=\Z^{n}\cap\overline{\conv(W\cap\mathcal{C})}^{\R^n}.$$
 Since $W\sub\R^{n}$ is defined over $\Q$, there is a pure monoid $\mathcal{F}\sub\Z^{n}$ such that $W=\conv(\mathcal{F})$. Hence, by Proposition \ref{convex}\ref{4}, we have $$A=\ri(W\cap K)=\ri\big(\conv(\mathcal{F})\cap \conv(\mathcal{C})\big)=\ri\big(\conv(\mathcal{F}\cap\mathcal{C})\big)\ .$$  Therefore there is a pure monoid $\mathcal{F}'\sub\Z^{n}$ such that $A^0=A\cup\{0\}=\conv(\mathcal{F}')$. Now, by Proposition  \ref{convex_0}\ref{1} and Proposition \ref{convex}\ref{4} again, we obtain
 $$\overline{\conv(A^0\cap\mathcal{C})}^{\R^n}=\overline{A^0\cap\conv(\mathcal{C})}^{\R^n}=\overline{A}^{\R^n}=$$
 $$=\overline{\ri(W\cap K)}^{\R^n}=\overline{W\cap K}^{\R^n}=\overline{\conv(W\cap\mathcal{C})}^{\R^n}.$$
 It follows that $\overline{\mathcal{D}}=\overline{W\cap\mathcal{C}}=\Z^{n}\cap\overline{A}^{\R^{n}}$.

 Now we can prove the claims of the statement.
 
 (i) Since $\mathcal{C}$ is prismal, $\overline{\mathcal{D}}=\overline{W\cap\mathcal{C}}$ is a finitely generated monoid.
 
 (ii) If $\dim(\mathcal{D})=\dim(K)$ then $\mathcal{D}=A^0\cap\mathcal{C}$, where $A=\ri(K)$ and $W=\gen{\mathcal{C}}$. By the preliminary part of the proof and by Proposition \ref{convex_0}\ref{1}, we have $\overline{\mathcal{D}}=\Z^{n}\cap\overline{A}^{\R^{n}}=\Z^{n}\cap\overline{K}^{\R^n}=\overline{\mathcal{C}}$.

 (iii) Let $0\neq\alpha\in \mathcal{D}$ and $\beta\in \mathcal{C}\setminus \overline{\mathcal{D}}$. Since $\overline{\mathcal{D}}=\Z^{n}\cap\overline{A}^{\R^{n}}$, we have $\beta\in K\setminus\overline{A}^{\R^n}$. The rest follows immediately from Theorem \ref{cone_decomp} and from the fact that $\mathcal{C}$ is pure.
 
 (iv) First note that, by Proposition \ref{convex}\ref{5}, we have $$\mathcal{D}= \mathcal{C}\cap A^0=\Z^{n}\cap\conv(\mathcal{C})\cap A^0 =\Z^{n}\cap A^0\ .$$ Now,  let $0\neq\alpha\in \mathcal{D}$ and $\gamma\in \overline{\mathcal{D}}=\Z^{n}\cap \overline{A}^{\R^n}$. Then $\alpha\in \ri(A)=A$ is an inner point of a convex set $A$ and $\gamma\in \overline{A}^{\R^n}$.  Since  $A^0$ is a cone, we therefore have $\alpha+\gamma\in A^0\cap\Z^{n}=\mathcal{D}$.
\end{proof}

\section{Every monoid associated to a finite tuple of semiring-generators of a parasemifield is prismal}\label{section 2}

Let now $S$ be a parasemifield. We use the canonical pre-order $\leq_{S}$ defined as $a\leq_{S} b$ if and only if $a=b$ or there exists $c\in S$ such that $a+c=b$. It is in fact an order (see e.g., \cite[Section 2]{vechtomov}). Note that it is preserved by addition, multiplication and   anti-preserved by inversion in $S$ (i.e., $a\leq_{S} b$ implies $a^{-1}\geq_{S} b^{-1}$ for all $a,b\in S$).

Let $A$ be the \emph{prime} subparasemifield of $S$, i.e., the smallest (possibly trivial) parasemifield contained in $S$. There are only two possibilities for $A$: either it is isomorphic to $\Q^+$, or it is trivial (i.e., it consists of a single element).

Let us now introduce the set $Q_S$ of all elements that are smaller than some element of $A$.
As   was already noticed by \cite{jezek,notes}, using $Q_S$ one can define a cone $\mathcal{C}_X(S)$ which plays a key role in the proofs. 

The set $$Q_{S}:=\{a\in S|(\exists q\in A)\ a\leq_{S} q\}$$ is a subsemiring of $S$. Clearly, $a,b\in Q_{S}$ for every $a,b\in S$ such that $a+b\in Q_{S}$.

We say that an $n$-tuple  $X=(x_1, \dots, x_n)$, where $x_i\in S$ for $i=1,\dots,n$, is \emph{a generating tuple of $S$ (considered as a semiring)} if $$S=\set{f(x_1,\dots,x_n)}{0\neq f\in\N[T_1,\dots,T_n]}$$ where $\N[T_1,\dots,T_n]$ is the semiring of polynomials over variables $T_1,\dots,T_n$ with non-negative integer coefficients. 

Let $X$ be such a generating tuple. For $\alpha=(a_{1},\dots,a_{n})\in\N_{0}^{n}$ we put $$\x^{\alpha}:=x_{1}^{a_{1}}\dots x_{n}^{a_{n}} $$ 
and we denote 
$$\mathcal{C}_X(S):=\set{\alpha\in\N_{0}^{n}}{\x^{\alpha}\in Q_{S}}$$
the \emph{corresponding monoid} assigned to $X$ and $S$.

 Obviously, $\mathcal{C}_X(S)$ is a submonoid of $(\N_{0}^n, +)$ and the semiring $Q_{S}$ is generated by the set $\set{\x^{\alpha}}{\alpha\in\mathcal{C}_X(S)}$.

 The goal of this section is to study the monoid $\mathcal{C}_X(S)$ and to show its prismality. The following two results establish some basic geometrical information about $\mathcal{C}_X(S)$. They were already used in \cite{jezek,notes}, but we include their short proofs for the sake of completeness.

\begin{proposition}[\cite{notes}, Lemma 4.6]\label{pure_preliminary}
If $a \in S$ and $n\in\N$ are such that $a^n \in Q_{S}$  then $a\in Q_{S}$.
\end{proposition}
\begin{proof}
Let $A$ be the prime parasemifield of $S$. If  $a^n\in Q_{S}$ then,  clearly, there are $u\in S$ and $q \in A$ such that $a^n + u = q^n$. Set $w = q^{n-1} + q^{n-2}a  + \cdots + qa^{n-2}  + a^{n-1}$. Then  we have 
$$aw+u = q^{n-1}a + q^{n-2}a^2  + \cdots + qa^{n-1}  + a^{n}+u =$$
$$=q^{n-1}a + q^{n-2}a^2  + \cdots + qa^{n-1}  + q^{n}  =qw\ .$$ 
Now, $ a+uw^{-1}  = q \in A$ and therefore $a\in Q_{S}$.
\end{proof}

\begin{corollary}\label{pure}
 Let $X=(x_1, \dots, x_n)$  be a generating tuple for a parasemifield $S$ (as a semiring). Then the associated monoid $\mathcal{C}_X(S)$ is pure.
\end{corollary}

  In order to show prismality of $\mathcal{C}_X(S)$, we need to first recall the classification of additively idempotent parasemifields, finitely generated as semirings \cite{l-groups}.

\begin{definition}\label{rooted_tree}
  Let us recall the notion of a rooted tree and an $\ell$-group (additively idempotent parasemifield, resp.) that is associated to it 
(for an explicit description with more details see \cite[beginning of Section 4]{l-groups}).

First note that to each lattice-ordered commutative group (an $\ell$-group for short) $G(\oplus, \vee, \wedge)$ corresponds an additively idempotent parasemifield $G(\vee, \oplus)$, and that to
describe the infimum and supremum operations $\wedge, \vee$ in an $\ell$-group $G$, it suffices to describe the corresponding ordering $\leq_G$.

A \emph{rooted tree} $(T, v_0)$ is a (finite, non-oriented) connected graph $T$
containing no cycles and having a specified vertex, the root $v_0$. 

Attach a copy of the group of integers $\Z = \Z_w$ to each vertex $w$ of $T$.   The $\ell$-group $G(T,v_{0})$ associated to $(T,v_{0})$ is an additive group that arises as a direct product of these groups. It remains to describe the partial order on $G(T,v_{0})$.
Let $e_w$ be the generator of the direct summand $\Z_w$.  The ordering $\leq_{G(T,v_{0})}$ on $G(T,v_{0})$ can be  expressed as follows:
 
 Consider $(T,v_{0})$ as a partially ordered set with the greatest element $v_{0}$ where the ordering $\preceq_{(T,v_{0})}$ is given by the graph $T$ that is considered as a Hasse diagram oriented downwards. 
 
First, assume that $T$ is a chain. Then  $\leq_{G(T,v_{0})}$ is defined as the lexicographical ordering on $G(T,v_{0})$ induced by the linear ordering $\preceq_{(T,v_0)}$ on $T$.

Now, consider the general case of a rooted tree. Then for $a,b\in G(T,v_{0})$ set $a \leq_{G(T,v_{0})} b$ if and only if $a \leq_{G(\widetilde{T},v_{0})} b$ for all possible extensions of $T$ into a chain $\tilde{T}$ with the same underlying set of vertices.

By the well-known correspondence of $\ell$-groups and commutative additively idempotent para\-se\-mi\-fields,  
$G(T,v_0)$ can be treated as an  additively idempotent parasemifield and $\leq_{G(T,v_{0})}$ is the natural ordering on this parasemifield (for example see \cite{W,WW}).
\end{definition}

\begin{proposition}\label{canonical_prismal}
  Let $(T, v_0)$ be a rooted tree and $S=G(T,v_0)$ be the additively idempotent parasemifield corresponding to it. Let $e_w$ have the same meaning as in Definition $\ref{rooted_tree}$ and let $X$ be a tuple of canonical generators of the parasemifield $S$ considered as a semiring, i.e., $X=(e_{w_1},-e_{w_1},\dots,e_{w_n},-e_{w_n})$, where $w_1,\dots,w_n$ are all the  pairwise different vertices of the graph $T$. 

Then the associated monoid $\mathcal{C}_{X}(S)$  is prismal.
\end{proposition}
\begin{proof}
First, assume that the rooted tree is a chain. In this case we see that $S$ is a parasemifield corresponding to an $\ell$-group $(\Z^{n},+)$ with the usual lexicographical ordering $\leq_{lex}$. 
In this case the canonical tuple is of the form  $X=(e_{1},-e_{1},\dots,e_{n},-e_{n})$, where $e_i\in\Z^n$ are the usual vectors of the standard basis (i.e., $e_1=(1, 0, 0, \dots, 0)$, $e_2=(0, 1, 0, \dots, 0)$ etc).

As the next step we show that the monoid $\mathcal{D}_{n}=\set{\alpha\in\Z^{n}}{\alpha\leq_{lex}0}$ is prismal. Clearly, the monoid $\mathcal{D}_{n}$ is pure. Now, we proceed by induction on $n\geq 1$. The case $n=1$ is obvious.  For the induction step, choose  a vector subspace $W$ of $\R^n$. Due to Proposition \ref{simplification}\ref{second}, we may consider that $V$ is defined over $\Q$. If $W\not\sub \{0\}\times \R^{n-1}$ then $\overline{\mathcal{D}_{n}\cap W}=\set{(a_1,\dots,a_n)\in W\cap\Z^{n}}{a_1\geq 0}$ and if $W\sub \{0\}\times \R^{n-1}$ then $\overline{\mathcal{D}_{n}\cap W}=\overline{(\{0\}\times\mathcal{D}_{n-1})\cap W}$. The induction step now follows easily and the monoid $\mathcal{D}_{n}$ is therefore prismal.

Now, $\mathcal{C}_{X}(S)=\set{(a_1,b_1,\dots,a_n,b_n)\in\N_{0}^{2n}}{(a_1-b_1,\dots,a_n-b_n)\in\mathcal{D}_n}$. In other words, $\mathcal{C}_{X}(S)=\N_{0}^{2n}\cap(\pi_{|\Z^{2n}})^{-1}(\mathcal{D}_n)$, where $\pi:\R^{2n}\to\R^n$, $\pi(a_1,b_1,\dots,a_n,b_n)=(a_1-b_1,\dots,a_n-b_n)$. By Theorems \ref{change}, \ref{epi_prismal} and Corollary \ref{pure}, it follows that $\mathcal{C}_{X}(S)$ is prismal.

In the case of a general rooted tree we obtain by the definition of $\leq_{G(T,v_0)}$  that
$$\mathcal{C}_{X}\big(G(T,v_0)\big)=\bigcap\Big\{\mathcal{C}_{X}\big(G(\widetilde{T},v_0)\big)\Big|\ \widetilde{T}\ \text{extends}\ T\ \text{to a chain} \Big\}\ .$$ 
The monoid $\mathcal{C}_{X}\big(G(T,v_0)\big)$ is therefore an intersection of prismal monoids (according to the first part of the proof) and thus, by Theorem  \ref{change}, $\mathcal{C}_{X}\big(G(T,v_0)\big)$ is prismal as well.
\end{proof}

\begin{theorem}[\cite{l-groups}, Theorem 4.1]\label{classify}
Let $S$ be an additively idempotent parasemifield, finitely ge\-ne\-ra\-ted as a semiring. Then $S$ is a (finite) product of
parasemifields of the form
${G(T_i,v_i)}$, where $(T_i, v_i)$ are rooted trees and $G(T_i,v_i)$ are associated additively idempotent parasemifields (or equivalently $\ell$-groups).
\end{theorem}

 Hence we explicitly understand the structure of additively idempotent parasemifields and of the corresponding monoids $\mathcal{C}_{X}(S)$ (by Proposition \ref{canonical_prismal}).
Given a general parasemifield, we will now show that its monoid is in fact the same as the monoid of some additively idempotent parasemifield.

\begin{proposition}\label{idempotent}
Define a congruence $\sim$ on $S$ by $x\sim y$ if and only if $xy^{-1}\in Q_S$ and $yx^{-1}\in Q_S$ for $x,y\in S$. Then $U=S_{/ \sim}$ is the largest factor-parasemifield of $S$ that is additively idempotent.

If $S$ is generated by a tuple $X=(x_1, \dots, x_n)$ as a semiring then  $U$ is generated as a semiring by the corresponding tuple $X'=(x'_{1}, \dots, x'_{n})$, where $x'_{i}=x_{i_{/ \sim}}$, and  $\mathcal{C}_{X}(S)=\mathcal{C}_{X'}(U)$. 
\end{proposition}
\begin{proof}
 First, we show that the relation $\sim$ is indeed a congruence of the parasemifield $S$. The only property that does not seem to be  obvious  is that $x\sim y$ implies $x+a\sim y+a$ for every $x,y,a\in S$. Assume therefore that $x\sim y$. Then $xy^{-1}\in Q_S$ and, consequently, there is $q$ in the prime subparasemifield $A$ such that $x\leq_{S}qy$. We can assume without loss of generality that $1\leq_{S}q$. Hence $x+a\leq_{S}qy+a\leq_{S}qy+qa=q(y+a)$  for every $a\in S$. We obtain that $(x+a)(y+a)^{-1}\in Q_{S}$ and similarly $(y+a)(x+a)^{-1}\in Q_{S}$. The relation $\sim$ is therefore indeed a congruence.
 
 Further, since $x(2x)^{-1}=2^{-1}\in Q_S$ and $2x(x)^{-1}=2\in Q_{S}$ we have $x\sim 2x$. Therefore $U=S_{/ \sim}$ is an additively idempotent parasemifield.  It remains to prove maximality. Let $\varphi:S\to T$ be a parasemifield homomorphism such that $T$ is additively idempotent. If we have $x\sim y$ then, as before, we know that $x\leq_{S}qy$ for some $q\in A$. Hence $\varphi(x)\leq_{T}\varphi(q)\varphi(y)=\varphi(y)$ and similarly, $\varphi(y)\leq_{T}\varphi(x)$. As the relation $\leq_{T}$ is an order in the additively idempotent parasemifield $T$, we get that $\varphi(x)=\varphi(y)$. This means that $U=S_{/ \sim}$ is the largest factor-parasemifield of $S$ which is additively idempotent.
 
 Finally, let $S$ be generated by a tuple $X=(x_1, \dots, x_n)$ as a semiring. Let $\pi:S\to U=S_{/ \sim}$ be the natural epimorphism. Clearly, $U$ is generated by the corresponding tuple $X'=(x'_{1}, \dots, x'_{n})$, where $x'_{i}=\pi(x_i)$. Since $U$ is additively idempotent, $\pi(A)$ is the prime subparasemifield of $U=\pi(S)$. 
 
 Let $\alpha\in\mathcal{C}_{X}(S)$. Then $\x^{\alpha}\in Q_S$ and therefore we have $\big(\x'\big)^{\alpha}=\pi(\x^{\alpha})\in\pi(Q_S)\sub Q_U$. We have obtained that $\mathcal{C}_{X}(S)\sub \mathcal{C}_{X'}(U)$.
 
 Let, on the other hand, be $\alpha\in \mathcal{C}_{X'}(U)$. Then $\pi(\x^{\alpha})=\big(\x'\big)^{\alpha}\leq_{U}1_U=\pi(1_S)$. Hence there is $b\in S$ such that $\pi(\x^{\alpha})+\pi(b)= \pi(1_S)$. It follows that $\x^{\alpha}+b\sim 1_S$ and there is $q\in A$ such that $\x^{\alpha}+b\leq q\cdot 1_S=q$. Thus $\x^{\alpha}\in Q_S$ and  $\alpha\in\mathcal{C}_{X'}(U)$. We have shown that $\mathcal{C}_{X'}(U)\sub \mathcal{C}_{X}(S)$. 
 
 Altogether we have proved that $\mathcal{C}_{X'}(U)= \mathcal{C}_{X}(S)$.
\end{proof}

  Finally, it remains to deal with the dependence of $\mathcal{C}_{X}(S)$ on the generating tuple $X$. We start with an easy lemma:

 \begin{lemma}\label{operations preserve prismality}
Let $X=(x_1, \dots, x_n)$ be a generating tuple for $S$ such that the monoid $\mathcal{C}_X(S)$ is prismal. Then for the generating tuple  $Y=(x_1, \dots, x_n,x_1)$, the monoid $\mathcal{C}_Y(S)$ is prismal as well.
\end{lemma}
\begin{proof}
Clearly, $\alpha=(a_1, \dots, a_{n}, a_{n+1})\in\mathcal{C}_Y(S)\sub\R^{n+1}$ if and only if
$$x_{1}^{a_{1}}x^{a_2}_{2}\cdots x_n^{a_n}x_1^{a_{n+1}}=x_1^{a_1+a_{n+1}}x^{a_2}_{2}\cdots x_n^{a_n}\in Q,$$
and this is equivalent to $(a_1+a_{n+1},a_2, \dots,a_{n})\in\mathcal{C}_X(S)\sub\R^{n}$. Set $\pi:\R^{n+1}\to\R^{n}$ as $$\pi(a_1, \dots, a_{n+1})=(a_1+a_{n+1},a_2, \dots,a_{n}).$$ Then, by  Theorems \ref{change} and \ref{epi_prismal}, $\mathcal{C}_{Y}(S)=\N_{0}^{n+1}\cap(\pi_{|\Z^{n+1}})^{-1}\big(\mathcal{C}_{X}(S)\big)$ is prismal.
\end{proof}

  Now we are ready to prove everything together and to show the main result of this section:

\begin{theorem}\label{prismality}
Let $X=(x_1, \dots, x_n)$  be  a generating tuple for a parasemifield $S$ (as a semiring). Then the associated monoid $\mathcal{C}_X(S)$ is prismal.
\end{theorem}
\begin{proof}
First we show that if $S$ is additively idempotent and finitely generated as a semiring then there is a tuple $Y=(y_1,\dots,y_k)$ of length at least two (i.e., $k\geq 2$)   such that
\begin{itemize}
 \item $y_1y_2=1_S$,
 \item the tuple $Y$ generates $S$ by using only the multiplication  and
 \item the associated monoid $\mathcal{C}_{Y}(S)$ is prismal.
\end{itemize}

First of all, assume that $S=G(T, v)$ for some rooted tree $(T, v)$. Then the monoid associated to the canonical generating tuple $Y$ is  prismal, by Proposition \ref{canonical_prismal}. Clearly, the tuple $Y$ generates $S$ using only multiplication (the corresponding $\ell$-group is generated by the tuple $Y$ only as a semigroup without using the inverse and infimum or supremum operations). In the case that $S$ is trivial, we can set $Y=(1_S,1_S)$.

  Now, if $S$ is an arbitrary additively idempotent parasemifield, then $S$ is a finite product of parasemifields $S_i=G(T_i, v_i)$, by Theorem \ref{classify}. If $Y_i$ are the corresponding canonical generating tuple of $S_i$  and we put $Y=\cup_i Y_i$ (i.e., the tuples are simply concatenated in some order) then $Y$ generates $S$ multiplicatively and $\mathcal{C}_{Y}(S)=\prod_{i}\mathcal{C}_{Y_i}(S_i)$. Since all the monoids $\mathcal{C}_{Y_i}(S_i)$ are prismal, the monoid $\mathcal{C}_{Y}(S)$ is prismal as well, by Theorem \ref{cartesian}. 

Further, we proceed with the general case of a parasemifield $S$ generated by the tuple $X=(x_1, \dots, x_n)$ as a semiring. By Corollary \ref{pure}, the associated monoid $\mathcal{C}_{X}(S)$ is pure. 
As in Proposition \ref{idempotent}, let  $U$ be the largest factor-parasemifield of $S$ which is additively idempotent and let   $X'=(x'_{1}, \dots, x'_{n})$ be the corresponding generating tuple of $U$. By Proposition \ref{idempotent}, we know that $\mathcal{C}_{X}(S)=\mathcal{C}_{X'}(U)\sub \N_0^n$.

By the previous part of the proof, there is a tuple $Y=(y_1,\dots,y_k)\sub U$ that generates $U$ multiplicatively, $y_1y_2=1_U$ and $\mathcal{C}_{Y}(U)$ is prismal. Elements of $X'$ may be expressed as monomials in $Y$, i.e., there is a $k\times n$ matrix $\mathbb{A}=(a_{i,j})$ with non-negative integer entries such that $x'_j=\prod_{i=1}^{k}y_i^{a_{i,j}}$ for every $j=1,\dots,n$. Clearly, for  $\alpha=(a_1,\dots,a_n)^{T}\in\N_{0}^{n}$ we have $\alpha\in\mathcal{C}_{X'}(U)$ if and only if $\mathbb{A}\alpha\in\mathcal{C}_{Y}(U)$. Let $\nu:\R^{n}\to\R^{k}$ be a linear map corresponding to the matrix $\mathbb{A}$. Clearly, $\mathcal{C}_{X'}(U)=\nu^{-1}\big(\mathcal{C}_{Y}(U)\big)$.

We can assume without loss of generality that $\nu$ is an embedding. If this is not the case, then some column (let us say the $j_0$-th column for $j_0\in\{1,\dots,n\}$) is linearly dependent on the other columns. Let $Y'$ be a generating tuple obtained from $Y$ by doubling the element $y_1$, i.e., $Y'=(y_1,\dots,y_k,y_1)$. Since $x'_{j_0}=\Big(\prod_{i=1}^{k}y_i^{a_{i,j}}\Big)\cdot(y_1 y_2)$, we can set 
\begin{displaymath}a'_{i,j}=\left\{
\begin{array}{lllll} 
a_{2,j_{0}}+1 & \text{if}\ i=2 & \& \ j=j_{0} \\
1 & \text{if}\  i=k+1 & \& \ j= j_{0}\\
0 & \text{if}\ i=k+1 & \& \ j\neq j_{0} \\
a_{i,j} & \text{otherwise}
\end{array}\right.\end{displaymath}
and the $(k+1)\times n$ matrix  $\mathbb{A}'=(a'_{i,j})$ expresses elements from $X'$ with help of $Y'$ similarly as before. The $j_0$-th column of $\mathbb{A}'$ is now independent on the others. We can repeat this process till we arrive at a matrix with rank $n$.

Now, the linear map $\nu:\R^{n}\to\R^{k}$ is an embedding and $\nu(\Z^{n})\sub\Z^{k}$. Therefore, by Proposition  \ref{alm_prism_equiv}, $\mathcal{C}_{X'}(U)=\nu^{-1}\big(\mathcal{C}_{Y}(U)\big)$ is prismal if and only if $\nu\big(\mathcal{C}_{X'}(U)\big)=\mathcal{C}_{Y}(U)\cap \mathrm{Im}(\nu)$ is prismal. Finally, since we know that $\mathcal{C}_{Y}(U)$ is prismal, it follows that $\mathcal{C}_{X'}(U)(= \mathcal{C}_{X}(S))$ is prismal as well. Therefore, any monoid associated to any tuple, that generates a parasemifield as a semiring, is prismal.
\end{proof}

\section{Every parasemifield that is finitely generated as a semiring is additively idempotent}\label{section 3}

 For a semiring $T$, the \emph{Grothendieck ring} $G(T)$ is defined in the same way as the Grothendieck group is defined for a (commutative) semigroup. Namely,  on the set $\widetilde{T}=T\times T$ we define  operations $\oplus$ and $\odot$ as $$(x,y)\oplus(x',y')=(x+x',y+y')\ \ \text{and}\ \ (x,y)\odot(x',y')=(xx'+yy',xy'+x'y)$$ and a relation $\approx$ as $$(x,y)\approx (x',y')\ \Leftrightarrow\ (\exists t\in T)\ x+y'+t=x'+y+t$$
for every $x,x',y,y'\in T$.  Now $\approx$ is a congruence on the semiring $(\widetilde{T},\oplus,\odot)$ and $G(T)=\widetilde{T}_{/\approx}$. For an element $x\in T$ denote $[x]$ the corresponding element in $G(T)$, i.e., we have a semiring homomorphisms $T\to G(T)$, defined as $x\mapsto [x]$.

The motivation behind the definition of $G(T)$ is that we would like to work with the difference ring $T-T$. Unfortunately, if $T$ is not additively cancellative, $T-T$ is not well-defined. 
To remedy this, one usually considers the Grothendieck ring $G(T)$ as defined above with the understanding that $(x,y)_{/\approx}\in G(T)$ should correspond to the difference $[x]-[y]$.

 Finally, note that for an element $u\in T$ it holds that $[u]=0$ in $G(T)$ if and only if $z=u+z$ for some element $z\in T$. We will use this easy observation several times in this section, especially in the proof of Theorem \ref{nilpotent}.

\begin{theorem}\label{equivalence_divisibility}
The parasemifield $S$ is additively idempotent if and only if the Grothendieck ring  $G(Q_{S})$ is trivial. 	
\end{theorem}
\begin{proof}
 If $S$ is additively idempotent then $Q_{S}$ is additively idempotent as well and therefore the Grothendieck ring $G(Q_{S})$ has to be trivial.

 For the opposite implication, assume that $G(Q_{S})$ is trivial. Then there is $t\in Q_{S}$, such that $t=1+t$. By the definition of $Q_S$, there are $u\in A$ and $s\in S$ such that $t+s=u$. Therefore we obtain $u=1+u$. This is a non-trivial equality within the prime subparasemifield $A$. It follows that $A$ can not be isomorphic to $\Q^+$ (that is additively cancellative) and, therefore, $A$ is trivial. This means that $S$ is additively idempotent.
\end{proof}

Let $X=(x_1,\dots,x_n)$ be a tuple that generates a parasemifield $S$ as a semiring. Let $\mathcal{M}\sub\N_{0}^{n}$ be a subset. Denote by $S_{X}(\mathcal{M})$ the additive subsemigroup of $S(+)$ that is generated by the set $\set{\x^{\alpha}}{\alpha\in \mathcal{M}}$ - recall that $\x^{\alpha}:=x_1^{a_1}\cdots x_n^{a_n}$ if $\alpha=(a_1, \dots, a_n)$. 

Note that, if $\mathcal{M}$ is a submonoid of $\N_{0}^{n}(+)$, then $S_{X}(\mathcal{M})$ is a semiring.

\begin{lemma}\label{decomp_lemma}
  Let $X=(x_1,\dots,x_n)$ be a tuple that generates a parasemifield $S$ as a semiring. Let $\mathcal{C}=\mathcal{C}_{X}(S)$ and  $\mathbf{D}(\mathcal{C})$ be the  decomposition  of $\mathcal{C}$ into monoids as in Corollary $\ref{monoid_decomp}$.
  
  Then for every $\mathcal{D}\in\mathbf{D}(\mathcal{C})$ and every $0\neq\alpha\in\mathcal{D}$ it holds:
  \begin{enumerate}

   \item\label{a1} If $u\in S_{X}(\mathcal{C}+\overline{\mathcal{D}})$ then $\x^{\alpha}u\in S_{X}(\mathcal{C})=Q_S$.
   \item\label{b1} If $u\in Q_S=S_{X}(\mathcal{C})$ then $\x^{\alpha}u\in S_{X}(\mathcal{F})$, where $$\mathcal{F}=\overline{\mathcal{D}}\cup\bigcup \set{\mathcal{E}\in \mathbf{D}(\mathcal{C})}{\dim(\mathcal{E})>\dim(\mathcal{D})}\ .$$
  \end{enumerate}
\end{lemma}
\begin{proof}
 (i) The element $u$ is a sum of elements of the form $\x^{\gamma+\beta}$, where $\gamma\in\overline{\mathcal{D}}$ and $\beta\in\mathcal{C}$. By Corollary \ref{monoid_decomp}\ref{iv}, we have $\alpha+\gamma\in \mathcal{D}\sub\mathcal{C}$.
Therefore, we obtain  that  $\x^{\alpha}u\in S_{X}(\mathcal{C})=Q_S$.

 (ii) The element $u$ is a sum of elements of the form $\x^{\beta}$, where $\beta\in\mathcal{C}$. Clearly, for $\beta\in \mathcal{C}\cap\overline{\mathcal{D}}$ we have $\alpha+\beta\in \overline{\mathcal{D}}\sub\mathcal{F}$ and therefore is $\x^{\alpha+\beta}\in S_{X}(\mathcal{F})$. If $\beta\in\mathcal{C}\setminus\overline{\mathcal{D}}$ then, by Corollary \ref{monoid_decomp}\ref{iii}, there is $\mathcal{E}\in\mathbf{D}(\mathcal{C})$ such that $\dim(\mathcal{E})>\dim(\mathcal{D})$  and $\alpha+\beta\in \mathcal{E}\sub\mathcal{F}$. Hence  $\x^{\alpha+\beta}\in S_{X}(\mathcal{F})$.
 
  Summing this up, we obtain that $\x^{\alpha}u\in S_{X}(\mathcal{F})$.
\end{proof}

For a ring $R$ let $\mathcal{N}(R)=\set{a\in R}{(\exists\ \ell\in\N)\ a^{\ell}=0}$ denote the \emph{nilradical of} $R$. 

\begin{remark}\label{divisible}
In a semiring $T$ an element $a\in T$ is called \emph{additively divisible} if for every $m\in\N$ there is $b_{m}\in T$ such that $a=m\cdot b_{m}$.

 Let us recall that in a finitely generated commu\-ta\-ti\-ve ring $R$ any additively divisible element $a\in R$ has to be trivial (see e.g., \cite[Examples 1]{divis}).

Note that since the prime subparasemifield $A$ of any parasemifield $S$ is either trivial or isomorphic to the positive rationals $\mathbb Q^+$, it is obvious that $A$ is additively divisible. It follows that every parasemifield $S$ is additively divisible. Likewise, each subsemiring $Q$ of $S$ containing $A$ is additively divisible; in particular, $1_S$ is additively divisible in $Q_S$.
\end{remark}

\begin{theorem}\label{nilpotent}
 Let $X=(x_1,\dots,x_n)$ be a tuple that generates a parasemifield $S$ as a semiring. Then for every $0\neq\alpha\in\mathcal{C}_{X}(S)$ the element $[\x^{\alpha}]$ is nilpotent in $G(Q_S)$.
\end{theorem}
\begin{proof}
 By Theorem \ref{prismality}, the monoid $\mathcal{C}:=\mathcal{C}_{X}(S)$ is prismal. Let $\mathbf{D}(\mathcal{C})$ be the  decomposition  of $\mathcal{C}$ into monoids as in Corollary \ref{monoid_decomp}. Every non-zero element $\alpha\in\mathcal{C}$ belongs into precisely one monoid $\mathcal{D}\in\mathbf{D}(\mathcal{C})$. That is why we prove our assertion by the downward induction on the dimension of monoids $\mathcal{D}$  in $\mathbf{D}(\mathcal{C})$, i.e., from the highest dimension $n_0=\dim(\mathcal{C})$ to the lowest one appearing in $\mathbf{D}(\mathcal{C})$.

 \vspace{2mm}
 
 We start with $n_0=\dim(\mathcal{C})$. Let  $\mathcal{D}\in\mathbf{D}(\mathcal{C})$ have the dimension $n_0$. By Corollary \ref{monoid_decomp}, parts \ref{i} and \ref{ii},  the monoid $\overline{\mathcal{D}}$ is finitely generated and $\mathcal{C}\sub \overline{\mathcal{D}}$. It follows that the semiring $Q'=S_{X}\left(\overline{\mathcal{D}}\right)$  is finitely generated and $1_S$ is an additively divisible element in $Q'$ (as discussed in Remark \ref{divisible}). The ring $G(Q')$ has these properties as well and therefore, by Remark  \ref{divisible}, it must be trivial. Hence $z=1_{S}+z$ in $Q'$ for some $z\in Q'$. 
 
 Now, pick $0\neq\alpha\in\mathcal{D}$. We have $\x^{\alpha}z=\x^{\alpha}+\x^{\alpha}z$.
 By Lemma  \ref{decomp_lemma}\ref{a1}, we obtain that $\x^{\alpha}z\in Q_S$ and therefore $[\x^{\alpha}]=0$ in $G(Q_S)$. In particular, $[\x^{\alpha}]$ is nilpotent in $G(Q_S)$.

 \vspace{2mm}
 
 For the induction step, assume that for the monoid $\mathcal{D}\in\mathbf{D}(\mathcal{C})$ holds that every element $[\x^{\delta}]$ is nilpotent in $G(Q_S)$ whenever the non-zero exponent $\delta\in \mathcal{C}$ lies in a monoid from $\mathbf{D}(\mathcal{C})$ of a bigger dimension than $\dim(\mathcal{D})$. Let  us still denote a few  auxiliary structures.
 \begin{itemize}
  \item Denote $Q''=S_{X}(\mathcal{C}+\overline{\mathcal{D}})$ the corresponding subsemiring of $S$.
  \item For an element $u\in Q''$ we denote $\llbracket u\rrbracket$ the corresponding element in $G(Q'')$ to distinguish it from the notation of analogous elements $[a]\in G(Q_S)$ with $a\in Q_S$.
  \item Denote $R=\ ^{G(Q'')}/_{\mathcal{N}\left(G(Q'')\right)}$ the quotient ring of $G(Q'')$ by its nilradical. 
  \item Let $\pi:G(Q'')\to R$ be the natural ring epimorphism.
  \item Let $T$ be the subring of $R$  generated by the set  $\set{\pi\big(\llbracket\x^{\beta}\rrbracket\big)}{\beta\in\overline{\mathcal{D}}}$. 
 \end{itemize}
   By Corollary \ref{monoid_decomp}\ref{i}, the monoid $\overline{\mathcal{D}}$ is finitely generated and therefore, the ring $T$ is finitely generated as well.

 Now, pick $0\neq\alpha\in \mathcal{D}$. Again, the element $1_S$ is divisible in $Q_S$, by Remark \ref{divisible}. Therefore we have a set of equalities $1_S=m\cdot z_m$, $m\in\N$, where $z_m\in Q_S$. It follows that $\x^{\alpha}=m\cdot \x^{\alpha}z_m$.
 
 Further, let us show that $\pi\big(\llbracket\x^{\alpha}z_m\rrbracket\big)\in T$. By Lemma \ref{decomp_lemma}\ref{b1}, the element $\x^{\alpha}z_m$ is a sum of elements of the form $\x^{\beta}$, where either $\beta\in\overline{\mathcal{D}}$ or $\beta\in \mathcal{E}$ for some monoid $\mathcal{E}\in\mathbf{D}(\mathcal{C})$ such that $\dim(\mathcal{E})>\dim(\mathcal{D})$. In the first case we clearly have that $\pi\big(\llbracket\x^{\beta}\rrbracket\big)\in T$. Let us therefore assume the latter case of $\beta$.
 Then, by the induction hypothesis, $[\x^{\beta}]$ is a nilpotent element in $G(Q_S)$. Since $Q_S\sub Q''$, there is a natural ring homomorphism $G(Q_S)\to G(Q'')$ and it follows that $\llbracket\x^{\beta}\rrbracket$ is a nilpotent element in $G(Q'')$ as well. Therefore $\pi\big(\llbracket\x^{\beta}\rrbracket\big)=0$ in $R$. In particular, $\pi\big(\llbracket\x^{\beta}\rrbracket\big)=0\in T$. Summing this up, we have proved that $\pi\big(\llbracket\x^{\alpha}z_m\rrbracket\big)\in T$.
 
Finally, as we have $\pi\big(\llbracket\x^{\alpha}z_m\rrbracket\big)\in T$ and $\pi\big(\llbracket\x^{\alpha}\rrbracket\big)=m\cdot \pi\big(\llbracket\x^{\alpha}z_m\rrbracket\big)$ for every $m\in\N$, the element $\pi\big(\llbracket\x^{\alpha}\rrbracket\big)\in T$ is now additively divisible  in the finitely generated ring $T$. By Remark \ref{divisible},  it follows that $\pi\big(\llbracket\x^{\alpha}\rrbracket\big)=0$ in $T$ and, of course, the same is true in $R=\ ^{G(Q'')}/_{\mathcal{N}(G(Q''))}$. Hence, there is $k\in\N$ such that  $\llbracket\x^{k\alpha}\rrbracket=0$ in $G(Q'')$.  Therefore, there is $u\in Q''$ such that $u=\x^{k\alpha}+u$.  Multiplying this equality by $\x^{\alpha}$ we obtain that  $\x^{\alpha}u=\x^{(k+1)\alpha}+\x^{\alpha}u$. By Lemma  \ref{decomp_lemma}\ref{a1}, it follows  that  $\x^{\alpha}u\in Q_S$.  This means that $[\x^{(k+1)\alpha}]=0$ in $G(Q_S)$. In other words, $[\x^{\alpha}]$ is  nilpotent in $G(Q_S)$.

This concludes our proof and, indeed, for every $0\neq\alpha\in\mathcal{C}_{X}(S)$ the element $[\x^{\alpha}]$ is nilpotent in $G(Q_S)$.
\end{proof}

\begin{theorem}\label{main-theorem}
 Every parasemifield that is finitely generated as a semiring is additively idempotent.
\end{theorem}
\begin{proof}
 We can assume without loss of generality  that the unity $1_S$ is one of the generators $x_1,\dots, x_n$, in particular $x_1=1_S$.
 By Theorem \ref{nilpotent}, the element $x_1=1_S$ is nilpotent in $G(Q_S)$.
 This implies that $G(Q_S)$ is trivial. By Theorem \ref{equivalence_divisibility}, the parasemifield $S$ is additively idempotent.
\end{proof}

We have proved in this way Theorem \ref{main thm}(a). As we explained in the introduction, the results of \cite{jk,a note} then imply also parts (b) and (c) of the theorem.

Let us conclude by pointing out the following surprising corollary of our results:

\begin{corollary}
 Let $S$ be a parasemifield that is finitely generated as a semiring. Then $S$ is finitely generated as a multiplicative semigroup.
\end{corollary}
\begin{proof}
 It follows immediately from Theorems \ref{main-theorem}, \ref{classify} and Proposition \ref{canonical_prismal}.
\end{proof}

\section{Acknowledgements}

We wish to thank the anonymous referee for a careful reading of the manusript and for several comments and suggestions that have helped improve the article.


\end{document}